\documentclass[10pt,a4paper]{article}
\usepackage{indentfirst,latexsym,bm}
\usepackage{graphicx}
\usepackage{amsmath,amsthm}
\usepackage{amssymb}
\usepackage{multicol}
\usepackage{ulem}
\usepackage{amsfonts}
\usepackage{mathrsfs}
\usepackage[pdfstartview=FitH]{hyperref}
 \newtheorem{thm}{Theorem}[section]
 \newtheorem{cor}[thm]{Corollary}
 \newtheorem{lem}[thm]{Lemma}
 \newtheorem{prop}[thm]{Proposition}
 
 \theoremstyle{definition}
 \newtheorem{dfn}[thm]{Definition}

 \newtheorem{rem}{Remark}

 \newcommand{\Rnum}{\mathbb{R}}
  \newcommand{\R}{\mathbb{R}}

 \newcommand{\mi}{\mathrm{i}}
 \newcommand{\Be}{\begin{equation}}
 \newcommand{\Ee}{\end{equation}}
 \newcommand{\Bs}{\begin{split}}
 \newcommand{\Es}{\end{split}}
 \newcommand{\PP}{\mathbb P}
  \newcommand{\Ll}{\langle}
   \newcommand{\Rr}{\rangle}
  \newcommand{\Bes}{\begin{equation*}}
 \newcommand{\Ees}{\end{equation*}}
 \newcommand{\dif}{\mathrm{d}}
  \newcommand{\tl}{\tilde}
    \newcommand{\E}{\mathbb E}
 
  \newcommand{\abs}[1]{\left\vert#1\right\vert}
   \newcommand{\set}[1]{\left\{#1\right\}}
 \newcommand{\norm}[1]{\left\Vert#1\right\Vert}

\allowdisplaybreaks

\title{ The Large Deviation Principle and Steady-state Fluctuation Theorem for the Entropy Production Rate of a Stochastic Process in Magnetic Fields}
\author{\rm
\noindent  Yong CHEN\\
\noindent \footnotesize School of Mathematics and Computing Science, Hunan
University of Science and Technology,\\
\noindent \footnotesize Xiangtan, Hunan, {\rm 411201},
P.R.China.\\
\rm  \noindent Hao GE\\
\noindent \footnotesize Beijing International Center for Mathematical Research\\
\noindent \footnotesize and Biodynamic Optical Imaging Center, Peking University\\
\noindent \footnotesize Beijing,  {\rm 100871}, P. R. China.\\
\noindent  Jie XIONG\\
\noindent \footnotesize Department of Mathematics, Faculty of Science and Technology, University of
Macau,\\
\noindent \footnotesize Taipa, Macau. jiexiong@umac.mo \\
\noindent  Lihu XU\\
\noindent \footnotesize Department of Mathematics, Faculty of Science and Technology, University of
Macau,\\
\noindent \footnotesize Taipa, Macau. lihuxu@umac.mo
}
\date{}

\begin{document}
 \maketitle

\begin{abstract}
Fluctuation theorem is one of the major achievements in the field of nonequilibrium statistical mechanics during the past two decades. Steady-state fluctuation theorem of sample entropy production rate in terms of large deviation principle for diffusion processes have not been rigorously proved yet due to technical difficulties. Here we give a proof for the steady-state fluctuation theorem of a diffusion process in magnetic fields, with explicit expressions of the free energy function and rate function. The proof is based on the Karhunen-Lo\'{e}ve expansion of complex-valued Ornstein-Uhlenbeck process.



\end{abstract}

\section{Introduction}

Nonequilibrium phenomena are widely present in physics, chemical reactions, cellular dynamics and biological systems \cite{NP1977}. To mathematically characterize the `distance' between these systems and equilibrium, the entropy production rate (EPR), as one of the fundamental concept in nonequilibrium thermodynamics, was first proposed at the macroscopic level \cite{Prigogine1961}. Concrete expressions of EPRs in various kinds of dynamics have already been discovered, such as in fluid dynamics, diffusion processes, chemical reactions and reaction-diffusion processes \cite{Eckart40a,Eckart40b,Bridgman40}. Bridgman clarified the important equation for ``increase of entropy in the region within a closed surface'' ($dS/dt$) in terms of the difference between ``entropy which has flowed out of the region across the surface'' and ``entropy generated by irreversible processes within the surface'' which is exactly the EPR \cite{Bridgman40}.

Since 1970s, there has been a growing fascination toward nonequilibrium phenomena in stochastic systems at steady state \cite{Hill1977,Hill1995}. Equipped by the mathematical tools in the field of stochastic process, especially the perspective of trajectory, the stochastic theory of nonequilibrium steady state, taking the stationary Markovian processes as the fundamental models, has been significantly developed in the past 40 years. A comprehensive, but rather mathematical monograph has already been published \cite{JQQ2004}, and also has been applied to concrete biophysical systems \cite{ZQQ2012,GQQ2012}.

The concept of time reversibility in Markovian dynamics, in a statistical sense, for the forward stationary dynamics and its time
reversal, is regarded as being equivalent to the concept of equilibrium state in thermodynamics. The idea can even trace back to Kolmogorov \cite{Kolmogorov1936,Kolmogorov1937}. Furthermore, the EPR is defined as the time-averaged relative entropy of the distributions for the forward and reversed processes \cite{JQQ2004}. Following this definition, a mesoscopic definition of the EPR along a single trajectory (called sample EPR) was proposed, as the Radon-Nikodym derivative of the distributions for the forward and reversed processes \cite{QM91,JQQ2004}.

In 1993, a breakthrough occurred in nonequilibrium statistical mechanics, when Evans, Cohen and Morriss  \cite{ECM} found in computer simulations that the sample EPR of
a steady flow has a highly nontrivial symmetry, which is called the fluctuation theorem in the mathematical theory put forward by Gallavotti and Cohen \cite{GC}. The fluctuation theorem gives a general formula valid in nonequilibrium systems, for the logarithm of the probability ratio of observing trajectories that satisfy or ``violate'' the second law
of thermodynamics \cite{FTrev2008}.

Kurchan \cite{Kur} later pointed out that the fluctuation theorem also holds for certain diffusion processes, and Lebowitz and Spohn \cite{LeSp} extended Kurchan's results to quite general Markov processes, formulated by the large deviation principle (LDP). Both of them used the mesoscopic definition of sample EPR. However, their proof is not mathematically rigorous. In 2003, Jiang and Zhang \cite{jiangzhang} rigorously proved the steady-state fluctuation theorem of the entropy production for Markovian jumping processes.

This paper is to partially fill in the gap that there is no rigorous steady-state fluctuation theorem of the EPR for diffusion processes. Our model \eqref{cp0} is a Langevin equation which governs the motion of a charged test particle in magnetic field undergoing collisions with
particles in medium, we refer readers to \cite[Sect. 11.3]{Bal97} for more details. This type of model has been intensively studied \cite{Bal97,cl,fgl}. We shall prove in this paper LDP as well as the associated steady-state fluctuation theorem of its sample EPR, and give the explicit expressions of the rate function and the free energy function.

It turns out rather difficult to rigorously prove the LDP and the steady-state fluctuation theorem of the sample EPR for diffusion processes. Recently, Kim had attempted, in \cite{Kim}, to prove the steady-state fluctuation theorem of the sample EPR for diffusion processes both in compact spaces as well as in $R^n$. Unfortunately the proof in \cite{Kim} was not rigorous and contained several serious gaps.

There are several classical approaches to studying LDP, among which the most famous are probably G\"{a}rtner-Ellis' theorem \cite{demzeit}, Varadhan's inverse theorem \cite{demzeit}, Kifer's criterion \cite{Kif90, JNPS15} and Wu's criterion \cite{Wu01}. With a close look at the sample EPR \eqref{eprform} below, we find that G\"{a}rtner-Ellis's theorem
are probably the most promising way, because one can use Girsanov transform to remove the stochastic integral in \eqref{eprform}.
However,  essential smoothness of the related Cram\'er function in G\"{a}rtner-Ellis' theorem is usually very hard to verify. Fortunately, for our equation \eqref{cp0}, we can apply the Karhunen-Lo\'{e}ve expansion of complex-valued Ornstein-Uhlenbeck process to get an explicit form for Cram\'er function.

It seems not easy to generalize our method to general high-dimensional Ornstein-Uhlenbeck processes. The main reasons are that there seem no general Karhunen-Lo\'{e}ve type expansions for high dimension Ornstein-Uhlenbeck processes such that the stochastic coefficients
are independent,
and that Eq. \eqref{cp0} under a Girsanov transform as \eqref{e:CMG} leads to a non-stationary or even non-ergodic process.

The organization of the paper is as the following. We give notations, the main results and a strategy of the proofs in Section 2, while study the Karhunen-Lo\'{e}ve expansion and related differential equations associated to \eqref{cp0} in Section 3. These results are used in Section 4 to prove
LDP as well as the steady-state fluctuation theorem of the sample EPR.

\section{Setting and main results}
\subsection{The stochastic process in magnetic fields and its EPR}

Suppose that $B(t)=(B_1(t),B_2(t))$ is a 2-dimensional real standard Brownian motion, we consider the following stationary complex Ornstein-Uhlenbeck process $\set{Z_t}$ arising in magnetic fields:

\begin{equation}\label{cp0}
      \dif Z_t =-e^{\mi \theta} Z_t\dif t+ \sqrt{2 \cos\theta}\dif {\zeta}^B_t,\quad t\ge 0,
\end{equation}
where $\theta\in (-\frac{\pi}{2},\,\frac{\pi}{2})$, and ${\zeta}^B_t=\frac{B_1(t)+\mi B_2(t)}{\sqrt2}$ is a complex Brownian motion.
Eq. \eqref{cp0} has a unique invariant measure $\mu$ with the form (see \cite[Eq.(11.20)]{Bal97} or \cite[Eq.(1.4)]{cl}):
\begin{equation}\label{meas}
\dif\mu(z) = \frac{1}{\pi }\exp\set{-(z_1^2+z_2^2)}\dif z_1\dif z_2=\rho(z)\dif z_1\dif z_2
\end{equation}
with $z=z_1+{\rm i} z_2$.
 $Z_t$ describes the motion of a test particle in magnetic fields which undergoes the collisions from
particles in medium \cite[Sect. 11.3]{Bal97}.


Since we can identify the metrics of $\mathbb{C}$ and $\R^2$, to study Eq. \eqref{cp0}, it is equivalent to investigate the following real Ornstein-Uhlenbeck process \cite{ash,ito}:
\begin{align}
\begin{bmatrix}
\dif X_1(t)\\ \dif X_2(t) \end{bmatrix}& =- \begin{bmatrix} \cos \theta & - \sin \theta\\  \sin \theta & \cos \theta \end{bmatrix}
\begin{bmatrix} X_1(t)\\ X_2(t) \end{bmatrix} \dif t + \sqrt{\cos\theta}
\begin{bmatrix} \dif B_1(t)\\ \dif B_2(t) \end{bmatrix}. \label{langevin}
\end{align}
From \cite{JQQ2004}, the sample EPR of Eq. \eqref{langevin} is
\begin{align}
e_p(t)(\omega)&:=\frac{1}{t}\log \frac{\dif \mathbf{P}_{[0,t]}}{\dif \mathbf{P}^{-}_{[0,t]}}(\omega) \notag\\
&=\frac{1}{t}\left[\frac{2\sin^2\theta}{\cos\theta}\int_{0}^t\,\abs{X(s)}^2 \dif s+\frac{2\sin\theta}{\sqrt{\cos\theta}}\int_0^t[X_2(s),-X_1(s)]\dif {B}(s)\right], \label{eprform}
\end{align}
it is also the sample EPR of Eq. \eqref{cp0}. The EPR is
\begin{equation*}
e_p:=\E^{\mu}[e_p(t)]=\frac{2 \sin^2 \theta}{\cos \theta},
\end{equation*}
where $\E^\mu$ is the probability expectation induced by the stationary process $X_t$. By ergodic theory \cite{JQQ2004}, we have
\begin{equation} \label{e:erg}
\lim_{t \rightarrow \infty} e_p(t)=e_p \ \ \ \ \ a.s..
\end{equation}
Naturally one may ask the convergence speed of \eqref{e:erg} and its deviation behavior. Hence, LDP and fluctuation theorem of sample EPR have their importance not only in physics but also in mathematics.

The process is reversible and thus $e_p(t) \equiv 0$ for all $t \ge 0$, if and only if $\theta=0$ \cite{QH2001}. From now on, we only consider the case
$$\theta \in \left(-\frac{\pi}2,0\right) \cup \left(0,\frac{\pi} 2\right). $$

\subsection{Cram\'er  function and LDP}
Let
\begin{equation*}
   \dif {W}(t)=\dif {B}(t) -\frac{2\lambda \sin\theta}{\sqrt{\cos\theta}}[X_2(t),\,-X_1(t)]'\dif t.
\end{equation*}
By the well-known Cameron-Martin-Girsanov theorem, ${W}(t)$ is a standard 2 dimensional Brownian motion under the measure $\PP_{W}$
which is uniquely determined by
\begin{equation} \label{e:CMG}
  \frac{\dif \PP_{W}}{\dif \PP_{B}}\big |_{\mathcal F_t}=\exp\set{-\frac{2\lambda^2 \sin^2\theta}{\cos \theta}\int_0^t \abs{X(s)}^2 \dif s +\frac{ {2}\lambda\sin\theta}{\sqrt{\cos\theta}}\int_0^t[X_2(s),-X_1(s)] \dif B(s)  } .
\end{equation}
Under the new measure $\PP^W$, then Eq. \eqref{langevin} is rewritten as
\begin{align}
\begin{bmatrix}
\dif Y_1(t)\\ \dif Y_2(t) \end{bmatrix}& =- \begin{bmatrix} \cos \theta & - (1+2\lambda)\sin \theta\\  (1+2\lambda) \sin \theta & \cos \theta \end{bmatrix}
\begin{bmatrix} Y_1(t)\\ Y_2(t) \end{bmatrix} \dif t + \sqrt{\cos\theta}
\begin{bmatrix} \dif W_1(t)\\ \dif W_2(t) \end{bmatrix} \nonumber\\
&=\begin{bmatrix} -r & d\\  -d & -r \end{bmatrix}
\begin{bmatrix} Y_1(t)\\ Y_2(t) \end{bmatrix} \dif t + \sqrt{\cos\theta}
\begin{bmatrix} \dif W_1(t)\\ \dif W_2(t) \end{bmatrix}.\label{langevin new}
\end{align}
$\mu$ is also the unique invariant measure of $Y(t)$.
Since the process $X(t)$ is stationary, so is $Y(t)$.

Define
$$\Lambda_t(\lambda)=\frac{1}{t}\log \E^{\mu} \exp\left\{t \lambda e_p(t)\right\},$$
by \eqref{e:CMG} and \eqref{langevin new} we have
\begin{align}\label{crm func}
  \Lambda_t(\lambda)
    &=\frac{1}{t}\log \E^{\mu}\exp\set{\frac{2\lambda\sin^2\theta}{\cos\theta}\int_{0}^t\,\abs{X(s)}^2 \dif s
    +\frac{ 2\lambda\sin\theta}{\sqrt{\cos\theta}}\int_0^t[X_2(s),-X_1(s)] \dif \vec{B}(s) }\nonumber\\
    &=\frac{1}{t}\log \E^{\mu}\exp\set{\frac{2\sin^2\theta}{\cos\theta}\lambda(1+\lambda)\int_{0}^t\,\abs{Y(s)}^2 \dif s}.
\end{align}
For a given $\lambda$, if
$\lim_{t \rightarrow \infty} \Lambda_t(\lambda)$ exists, we denote
$$\Lambda(\lambda):=\lim_{t \rightarrow \infty} \Lambda_t(\lambda)$$
and
$$\mathcal D_{\Lambda}:=\{\lambda \in \R: \Lambda(\lambda)<\infty\}.$$
$\Lambda$ is called Cram\'er function \cite{demzeit}.

Let us next recall the definition of LDP \cite{demzeit}.
Let $\mathcal X$ be a complete separable metric space, $\mathcal B(\mathcal X)$ the Borel $\sigma$-field of $\mathcal X$, and $\{X_t\}_{t\ge0}$ a family of  stochastic processes valued in $\mathcal X$.

\begin{dfn}
 \label{d:LDP}
 $\{X_t\}_{t\ge0}$ satisfies a large deviation principle (LDP) if there exist a family of positive numbers $\{h(t)\}_{t\ge0}$ which tends to $+\infty$ and a function $I(x)$ which maps $\mathcal X$ into $[0,+\infty]$ satisfying the following conditions:
  \begin{itemize}
  \item[(i)]  for each $l<+\infty$, the level set $\{x:I(x)\le l\}$ is compact in $\mathcal X$;
  \item[(ii)] for each closed subset $F$ of $\mathcal{X}$,
              $$
                \limsup_{t\rightarrow \infty}\frac1{h(t)}\log\mathbb P (X_t\in F)\leq- \inf_{x\in F}I(x);
              $$
   \item[(iii)]for each open subset $G$ of $\mathcal{X}$,
              $$
               \liminf_{t\rightarrow \infty}\frac1{h(t)}\log\mathbb P(X_t\in G)\geq- \inf_{x\in G}I(x),
              $$
   \end{itemize}
here $h(t)$ is called the  speed function and $I(x)$ is the rate function.
\end{dfn}

\subsection{The main results and the strategy for their proofs}
Our main results are the following two theorems, the first one is about large deviation, while the second is the steady-state fluctuation theorem.
\begin{thm}[LDP]\label{main_thm}
The sample EPR $e_p(t)$ defined by \eqref{eprform}, satisfies an LDP with $h(t)=t$ in Definition \ref{d:LDP}. The rate function is
\begin{equation}\label{rate func}
  I(x)=\left\{
      \begin{array}{ll}
     - \frac12(1-\sqrt{\frac{x^2(1+2rc)c}{x^2c+2r}})x+\frac{r}{\pi}\int_0^{\infty}\log (1-\frac{x^2c^2-1}{(2r+x^2c)(1+y^2)c})\dif y,  & x\ge 0 ,  \\
      -\frac12(1+\sqrt{\frac{x^2(1+2rc)c}{x^2c+2r}})x+\frac{r}{\pi}\int_0^{\infty}\log (1-\frac{x^2c^2-1}{(2r+x^2c)(1+y^2)c})\dif y, &x<0,
      \end{array}
\right.
\end{equation}
where $c=\frac{\cos\theta}{2\sin^2\theta}$.
\end{thm}
\begin{thm}[Fluctuation theorem]\label{main_thm 2}
The Cram\'{e}r's function (or say: free energy function) of $\set{e_p(t):\,t\ge 0} $ is
\begin{equation}\label{crm func 000}
   \Lambda(\lambda)=\left\{
      \begin{array}{ll}
     -F(\ell), &\quad \text{for } \lambda \in \left(-\frac 12-\frac 12 \sqrt{1+\frac{r \cos \theta}{\sin^2 \theta}},-\frac 12+\frac 12 \sqrt{1+\frac{r \cos \theta}{\sin^2 \theta}}\right) ,  \\
     \infty , &\quad \text{for } \lambda \notin \left(-\frac 12-\frac 12 \sqrt{1+\frac{r \cos \theta}{\sin^2 \theta}},-\frac 12+\frac 12 \sqrt{1+\frac{r \cos \theta}{\sin^2 \theta}}\right),
      \end{array}
\right.
\end{equation}
where
\begin{equation}\label{lamd}
   \ell=\frac{2\sin^2\theta}{\cos\theta}\lambda(1+\lambda),
\end{equation}
and for $\ell<\frac{r}{2}$,
\begin{equation}\label{fell}
  F(\ell)= \int_0^{\infty} \log(1- \frac{2\ell r }{r^2+  \pi^2 y^2} )\dif y.
\end{equation}
Furthermore,
$\Lambda(\lambda)$ and $I(x)$ have the following properties:
\begin{equation}\label{fluc theorem}
   \Lambda(\lambda)=\Lambda(-(1+\lambda)),\quad \forall \lambda\in \Rnum^1,\quad I(x)=I(-x)-x,\quad \forall x\in \Rnum.
\end{equation}
\end{thm}

We shall prove the above theorems by finding an explicit form of Cram\'er function $\Lambda$ and checking that it satisfies the three conditions in Gartner-Ellis' Theorem. We shall use the following strategy to figure out the explicit form of Cram\'er function:

(1). We rewrite Eq. \eqref{langevin new} as a complex stochastic equation by setting $Z_t=Y_1(t)+\mi Y_2(t)$.

(2). Under the framework of complex-normal system, by Karhunen-Lo\'{e}ve expansion we write $(Z_t)_{0 \le t \le T}$ under a special orthonormal basis $\{e_k\}_{k \ge 1}$ of
$L^2([0,T];\mathbb C)$ as $Z_t=\sum_{k \ge 1} w_k e_k(t)$
such that $\{w_k\}_{k \ge 1}$ are independent normal-distributed sequence. Hence,
\
\Be
\E \exp\{\int_0^T |Z_s|^2 \dif s\}=\prod_{k \ge 1} \E e^{|w_k|^2}
\Ee

(3). To find the special orthonormal basis $\{e_k\}_{k \ge 1}$ in (2), we need to solve eigenvalues problem of a trace class integral operator
  \begin{equation*}
     Kf(s)=\int_0^T R(t,s)f(s)\dif t,\quad  f \in L^2([0,T];\mathbb C)
       \end{equation*}
  with $R(t,s)$ is a kernel satisfying $R(s,t)=\overline{R(t,s)}$. This problem turns out to look for solutions of a complex ordinary differential equations.

(4). We demonstrate the argument in  (3) when $\{Z_t: 0 \le t \le T\}$ is a real normal system. To look for the special orthonormal basis
$\{e_k\}_{k \ge 0}$ in $L^2([0,T];\mathbb R)$ such that
$Z_t=\sum_{k \ge 1} w_k e_k(t)$ and $w_1,w_2,...$ are independent, we only need to check
$$\E\Ll Z, e_i\Rr \Ll Z, e_j\Rr=0 \ \ \  \ \ \forall  \ i \neq j,$$
i.e.,
\Be \label{e:KReal}
\int_0^T \int_0^T \E(Z_sZ_t) e_i(s) \dif s e_j(t) \dif t=0.
\Ee
Let $R(t,s)=\E(Z_sZ_t)$ and $Ke_i(t)=\int_0^T R(t,s) e_i(s) \dif s$, as long as $e_i$ is an eigenfuction of $K$ with $\lambda_i$, i.e.,

\Be \label{e:KReEi}
K e_i=\lambda_i e_i,
\Ee
(15) immediately follows. Eq. \eqref{e:KReEi} is equivalent to solve an ordinary differential equation \cite[Chapter 1]{kuo}.

However, the expansion for a complex-normal system is much more
complicated than the above real normal case, we need to use Karhunen-Lo\'{e}ve expansion.


\section{Karhunen-Lo\'{e}ve expansion of complex-valued Ornstein-Uhlenbeck process}
\subsection{Some preliminary of Karhunen-Lo\'{e}ve expansion}
In this section, we adopt the formulation and notations of \cite{ito} and \cite[Chapter 1]{ash}.
\begin{dfn}
  Let $\xi=\xi_1+\mi \xi_2$ be a complex random variable. If $\xi_1$ and $ \xi_2$ are independent real random variables and subject to the same distribution
  $N(0,\frac{a}{2})$ with $a\ge 0$, $\xi$ is called a complex-normal random variable and we denote it by $ \mathcal{CN}(0,a)$. A system of complex random variables $\Xi=\set{\xi_{\lambda},\lambda\in \Lambda}$
  is called complex-normal if every linear combination of $\Xi$ is a complex-normal variable.
\end{dfn}
\begin{rem}
  The above definition is from \cite{ito}, which is more strict than the complex jointly Gaussian random variables given in \cite[p19]{ash}. This difference
ensure the independent property in the Karhunen-Lo\'{e}ve expansion.
\end{rem}

We cite the following well-known facts about the Karhunen-Lo\'{e}ve expansion \cite[Theorem 1.4.1]{ash} for further use.
\begin{thm}\label{kl exp}
  Let $\set{X_t,a\le t\le b}$, $a,b$ finite, be a complex-valued $L^2$ process with zero mean and continuous covariance $R(\cdot,\cdot)$. The integral operator $K$ associated with $R$ is
  \begin{equation}\label{kop}
     Kf(t)=\int_a^b R(t,s)f(s)\dif s,\quad s\in[a,b],\quad f\in L^2([a,b], \mathbb{C}).
  \end{equation}
  Let $\set{e_k,k=1,2,\dots}$
  be an orthonormal basis for the space spanned by the eigenfunctions of the nonzero eigenvalues of the operator $K$, with $e_k$
  taken as an eigenfunction corresponding to the eigenvalue $\lambda_k$. Then
  \begin{equation} \label{e:XW}
     X_t=\sum_{k\ge 1} w_k e_k(t),\quad a\le t\le b,
  \end{equation}
  where $w_k=\int_a^b X_t\bar{e}_k(t)\dif t$, and the $w_k$ are orthogonal random variables with
  \begin{equation} \label{e:gam}
  \E[w_k]=0, \E[\abs{w_k}^2]=\gamma_k.
  \end{equation}
  The series converges in $L^2$, uniformly in $t$. If the original process $X_t$ is a complex-normal system, then the variables $w_k$ are complex-normal and are stochastically independent, which implies that the above series also converges almost surely.
\end{thm}

In case of real normal system, the following type corollary is well-known ( \cite{demart}). Now we give a version for the complex-normal system.
Note that $\gamma_k \ge 0$  for all $k \ge 1$ in Theorem \ref{kl exp}. Without loss of generality, we assume that $\gamma_1 \ge \gamma_2 \ge ....$
\begin{cor}\label{cor3-0}
Let the complex valued process $\{X_t: a \le t \le b\}$ be
complex normal and let $\{\gamma_k: k \ge 1\}$ be the eigenvalues defined by \eqref{e:gam} with order $\gamma_1 \ge \gamma_2 \ge ....$.
Then, for all $z \in \mathbb C$, we have
\begin{equation}\label{eexp}
    \E[\exp\set{z \int_a^b \abs{X(t)}^2\dif t}]=\left\{
      \begin{array}{ll}
     \prod_{k=1}^{\infty} \frac{1}{1-z\gamma_k}, &\quad \text{for } \Re{z}<\frac{1}{\gamma_1} ,  \\
     \infty , &\quad \text{for } \Re{z}\ge \frac{1}{\gamma_1},
      \end{array}
\right.
\end{equation}
where $\Re{z}$ is the real part of $z$.

\end{cor}
\begin{proof}
It is easy to see from \eqref{e:gam} that $\gamma_k \ge 0$ for all $k \ge 1$.
  Since $w_k,\,k\ge 1$ have distribution $\mathcal{CN}(0,\gamma_k)$ and are stochastically independent, we have
\begin{equation*}
   \E[\exp\set{z D^2}]=\prod_{k=1}^{\infty} \E\big[\exp\set{z \abs{w_k}^2}\big].
\end{equation*}
On the other hand, it is easy to check
\begin{equation*}
   \E\big[\exp\set{z \abs{w_k}^2}\big]=\set{ \E\big[\exp\set{z {\Re(w_k)}^2}\big]}^2
   =\left\{
      \begin{array}{ll}
     \frac{1}{1-z \gamma_k}, &\quad \text{for } \Re{z}<\frac{1}{\gamma_k} ,  \\
     \infty , &\quad \text{for } \Re{z}\ge \frac{1}{\gamma_k}.
      \end{array}
\right.
\end{equation*}
Combining the previous two relations immediately yields \eqref{eexp}.
\end{proof}
\subsection{The covariance function and the associated integral operator $K_T$}
We rewrite \eqref{langevin new} as a complex-valued Ornstein-Uhlenbeck process
\begin{equation}\label{cp}
  \dif Z_t=-\alpha Z_t\dif t+ \sqrt{2\cos\theta}\dif \zeta_t,
\end{equation}
where $  \alpha= r+\mi d$ and $\zeta_t=\frac{W_1(t)+\mi W_2(t)}{\sqrt2}$.
Note that $Z_0\sim \mu$ and $Z_t$ is a stationary process.
It is easy to see that $Y_t$ and $Z_t$ have the same distribution since we can identify $\R^2$ with $\mathbb C$
\cite{ash,ito}. Hence, by \eqref{crm func} we have
\
\Be   \label{crm func 1}
\Lambda_t(\lambda)=\frac{1}{t}\log \E^{\mu}\exp\set{\frac{2\sin^2\theta}{\cos\theta}\lambda(1+\lambda)\int_{0}^t\,\abs{Z(s)}^2 \dif s}.
\Ee
It is easy to check
\begin{equation}\label{solv}
   Z_t=e^{- \alpha t}\left(Z_0+ \sqrt{2\cos\theta} \int_{0}^t  e ^{ \alpha s}\, \dif \zeta_s \right).
\end{equation}
\begin{lem}
$\{Z_t: 0 \le t \le T\}$ is a complex-normal system and has a continuous covariance function $R:[0,T] \times [0,T] \rightarrow \mathbb C$ with the form:
\
\Be
\begin{split}
& R(s,t)=e^{-\alpha(s-t)},  \ \ \ \ 0 \le t \le s \le T; \\
&R(s,t)=e^{\bar \alpha(s-t)},  \ \ \ \ 0 \le s \le t \le T.
\end{split}
\Ee
\end{lem}

\begin{proof}
Let $n \in \mathbb{N}$, take a sequence of points $0, \frac{T}{2^n}, \frac{T}{2^{n-1}},...,\frac{(2^n-1)T}{2^n},T$ and define
\
\Bes
Z^n_t=e^{-\alpha t} Z_0+e^{-\alpha t}\sum_{k=1}^{2^n}  \int_{\frac{(k-1)T}{2^n} \wedge t}^{\frac{kT}{2^n} \wedge t} e^{\alpha s} \dif \zeta_s.
\Ees
It is easy to see that $Z^n_t$ converges to $Z_t$ in probability for every $t \in [0,T]$.

Since $Z_0$ are independent of $\zeta_t$,
$$\left\{Z_0, \int_{0}^{\frac{T}{2^n}} e^{\alpha s} \dif \zeta_s, \int_{\frac{T}{2^n}}^{\frac{T}{2^{n-1}}} e^{\alpha s} \dif \zeta_s, ......, \int_{\frac{(2^n-1)T}{2^n}}^{T} e^{\alpha s} \dif \zeta_s\right\}$$
is a complex-normal system.
By \cite[Theorem 2.2]{ito}, we immediately obtain that $\{Z_t: 0 \le t \le T\}$ is a complex-normal system.

It is well known that the covariance function associated to $Z_t$ is
$$R(s,t)=\E[Z_s \overline Z_t].$$
When $h=s-t \ge 0$, by stationarity of $Z_t$ and the easy fact $\E[Z_t|\,Z_0]=e^{-\alpha t} Z_0$, we have
\begin{align*}
 R(s,t)&=\E[ Z_{s}\overline{Z_t}]=\E[ Z_{h}\overline{Z_0}]=\E[ \E[Z_{h}|\,Z_0]\overline{Z_0}]=e^{-\alpha h}\E[Z_0\overline{Z_0}] =e^{-\alpha h}.
\end{align*}
When $h=s-t<0$, similarly we have
\begin{align*}
 R(s,t)=e^{\bar{\alpha} h}.
\end{align*}
\end{proof}



Define an operator $K_T$ on $L^2([0,T];\mathbb C)$ by (see \cite[p38]{ash})
\begin{equation}\label{kopera}
  K_T f(s)=\int_0^T R(s,t) f(t) \dif t,\quad f\in L^2([0,T]),
\end{equation}
where
\begin{equation} \label{kop1}
   R(s,t)=\left\{
      \begin{array}{ll}
      e^{-\alpha (s-t)},&\quad s\ge t,  \\
      e^{ \bar{\alpha} (s-t)} ,&\quad s< t,
      \end{array}
\right.
\end{equation}
Since $R(t,s)$ is continuous on $[0,T]\times [0,T]$, it is well-known that $K_T$ is a positive self-adjoint trace class operator (see \cite[Theorem 3.9]{simon}), therefore $K_T$ has
discrete real eigenvalues (counting \ the \ multiple \ eigenvalues)
$$\gamma_{T,1} \ge ... \ge \gamma_{T,n} \ge ... $$ such that
$${\rm Tr}(K_T)=\sum_{i \ge 1} \gamma_{T,i}=\int_0^T R(t,t) \dif t=T<\infty.$$


\begin{cor}\label{cor3}
We have
\begin{equation}\label{eexp}
    \E\exp\set{z \int_0^T |Z_t|^2 \dif t}=\left\{
      \begin{array}{ll}
     \prod_{k=1}^{\infty} \frac{1}{1-z\lambda_k}, &\quad \text{for } \Re{z}<\frac{1}{\gamma_{T,1}} ,  \\
     \infty , &\quad \text{for } \Re{z}\ge \frac{1}{\gamma_{T,1}}.
      \end{array}
\right.
\end{equation}

\end{cor}
\begin{proof}
Since $(Z_t)_{0 \le t \le T}$ is a complex-normal system, Corollary \ref{cor3-0} immediately yields the desired result.
\end{proof}

\subsection{An ordinary differential equation associated to the eigenvalues problem of $K_T$}
The previous proposition only shows that $K_T$ is uniformly bounded, to apply Theorem \ref{kl exp}, we need to find the eigenvalues $\{\gamma_{i,T}\}_{i \ge 1}$ of $K_T$ and get more information about eigenvalues. The following lemma is important for solving Eq. \eqref{kfequa} below.
\begin{lem}\label{boundd}
For every $T \in (0,\infty)$, the operator norm of $K_T$ satisfies that
\begin{equation}\label{bdd}
  \norm{K_T}\le \frac{2}{r}.
\end{equation}
Thus, the eigenvalues of $K_T$  satisfy
\Be  \label{e:GamBou}
\gamma_{T,i} \le \frac 2 r, \ \ \ \ \forall \ i,\quad \forall T>0.
\Ee
\end{lem}

\begin{proof}
Without loss of generality, we extend the functions $R$ and $f$ to be new functions $\tl R$ and $\tl f$ as below:
$$\tl f(t)=f(t), \ \ \ t \in [0,T]; \ \ \ \ \  \tl f(t)=0, \ \ \ t \notin [0,T]; $$
$$\tl R(s,t)=R(s,t), \ \ \ (s,t)  \in [0,T]^2; \ \ \ \ \  \tl R(s,t)=0, \ \ \ (s,t) \notin [0,T]^2. $$
Then,
\
\Be
\begin{split}
\int_0^T \left|\int_0^T R(s,t) f(t) \dif t\right|^2 \dif s&=\int_{-\infty}^\infty \left|\int_{-\infty}^\infty \tl R(s,t) \tl f(t) \dif t\right|^2 \dif s  \\
&=\int_{-\infty}^\infty \left|\int_{-\infty}^\infty \tl R(s,s-u) \tl f(s-u) \dif u\right|^2 \dif s  \\
& \le \int_{-\infty}^\infty \left|\int_{-\infty}^\infty e^{-r|u|} |\tl f(s-u)| \dif u\right|^2 \dif s  \\
& \le \int_{-\infty}^\infty \frac 2r\left(\int_{-\infty}^\infty e^{-r|u|} |\tl f(s-u)|^2 \dif u\right) \dif s  \\
& = \frac 2r\int_{-\infty}^\infty \left(\int_{-\infty}^\infty |\tl f(s-u)|^2 \dif s\right) e^{-r|u|} \dif u  \\
& \le \frac 4{r^2} \|f\|^2_{L^2}.
\end{split}
\Ee
where the second inequality is by Jessen inequality. Hence, \eqref{bdd} is proved.
\end{proof}
\begin{rem}
  It follows from Proposition~\ref{pop46} that for every $T>0$,$ \norm{K_T}= \frac{2}{r} $.
\end{rem}

Following the spirit in \cite[p113]{ru}, we derive a differential equation from which all the eigenvalues and eigenfunctions of $K_T$ can be found. More precisely, we have
\begin{prop}
 Let the operator $K_T$ be \eqref{kopera}-\eqref{kop1}. Then ${\rm Ker}(K_T)=\set{0}$ and if $ K_T f=\gamma f $ with $\gamma\neq 0$ then
$f$ is a solution on $[0,T]$ of the differential equation
 \begin{equation}\label{kfequa}
    \gamma f'' + 2\mi d \gamma f' +(2r- \gamma \abs{\alpha}^2) f=0,
 \end{equation}
subject to the boundary conditions
\begin{align}\label{kf}
\left\{
      \begin{array}{ll}
      \bar{\alpha} f(0)-f'(0)=0, \\
      \alpha  f(T)+f'(T)=0.
      \end{array}
\right.
\end{align}
\end{prop}
\begin{proof}
Our proof follows the spirit in \cite[p113]{ru}.
For any $f\in L^2([0,T];\mathbb C)$, it follows from (\ref{kopera}) that
\begin{equation}\label{kf11}
  K_T f(s)=e^{-\alpha s}\int_0^s e^{ \alpha t}f(t)\dif t+ e^{ \bar{\alpha} s}\int_s^T e^{-\bar{\alpha} t}f(t)\dif t.
\end{equation}
Thus, $K_T f(s),\,s\in[0,T]$ is absolutely continuous. By differentiating both sides of (\ref{kf11}) with respect to $s$, we obtain
\begin{equation}\label{kfpri}
 ( K_T f)'=-\alpha e^{-\alpha s}\int_0^s e^{ \alpha t}f(t)\dif t+  \bar{\alpha} e^{ \bar{\alpha}s}\int_s^T e^{-\bar{\alpha} t}f(t)\dif t.
\end{equation}
Similarly, $(K_T f)'(s),\,s\in[0,T]$ is absolutely continuous and we have that
\begin{equation}\label{kfpri2}
 ( K_T f)''= \alpha^2 e^{-\alpha s}\int_0^s e^{ \alpha t}f(t)\dif t
 +  \bar{\alpha}^2 e^{ \bar{\alpha} s}\int_s^T e^{-\bar{\alpha} t}f(t)\dif t-(\alpha+\bar{\alpha})f(s).
\end{equation}
It follows from (\ref{kf11}) and (\ref{kfpri}) that
\begin{equation}\label{mm}
   e^{-\alpha s}\int_0^s e^{ \alpha t}f(t)\dif t=\frac{\bar{\alpha}K_T f-( K_T f)' }{\bar{\alpha}+\alpha},\quad
    e^{ \bar{\alpha} s}\int_s^T e^{-\bar{\alpha} t}f(t)\dif t=\frac{ \alpha K_T f+( K_T f)' }{\bar{\alpha}+\alpha}.
\end{equation}
Substituting it into (\ref{kfpri2}), we have that
\begin{equation}\label{bbb}
  ( K_T f)''=\abs{\alpha}^2 K_T f+(\bar{\alpha}-\alpha)(K_T f)'-(\alpha+\bar{\alpha})f.
  \end{equation}
Since $\alpha+\bar{\alpha}>0 $, $ Kf(t)=0 $ and (\ref{bbb}) imply that $f(t)\equiv 0$ in $L^2([0,T]; \mathbb C)$.

If $K_T f=\gamma f $ with $\gamma\neq 0$ then substituting $s=0$ and $s=T$ to (\ref{mm}) implies that
$\bar{\alpha} f(0)-f'(0)=0,\, \alpha  f(T)+f'(T)=0$. And (\ref{bbb}) implies that (\ref{kf}) hold.
\end{proof}

Note that $\alpha=r+\mi d$, it follows from Eq. (\ref{kfequa}) that the general solution of equation (\ref{kfequa}) is
\begin{equation}\label{solutionkr}
 f(s)=c_1 e^{-\mi (d+\omega)s}+c_2 e^{-\mi (d-\omega)s},
\end{equation}
where $c_1,c_2$ are constants and $\omega=\sqrt{\frac{2r}{\gamma}-r^2}$. By \eqref{bdd},
\Be \label{e:ome}
\omega=\sqrt{\frac{2r}{\gamma}-r^2}\ge 0
\Ee

 The boundary condition (\ref{kf}) gives
\begin{equation}\label{bdd0}
\left\{
      \begin{array}{ll}
        (r +\mi  \omega )c_1 +(r -\mi \omega) c_2=0,  \\
     (r -\mi \omega) e^{-\mi \omega T}c_1+  (r +\mi \omega)e^{ \mi \omega T} c_2  =0.
      \end{array}
\right.
\end{equation}
In order to have constants $c_1$ and $c_2$ such that $c_1^2+c_2^2\neq 0$,
we need
$$\frac{r+\mi\omega}{(r-\mi\omega) e^{-\mi \omega T}}=\frac{r-\mi\omega}{(r+\mi\omega) e^{\mi \omega T}},$$
i.e.
\begin{equation}\label{ew}
e^{2\mi\omega T}=(\frac{r-\mi \omega}{r+\mi \omega})^2,
\end{equation}
where $\omega$ is the unknown in Eq. \eqref{ew}. Eqs. \eqref{ew} and \eqref{e:ome} imply that $\omega$ satisfies
\begin{equation}\label{cot2}
  \frac{\omega}{r}=\cot\frac{\omega T}{2}, \ \ \ \ \omega \ge 0,
\end{equation}
or
\begin{equation}\label{tan2}
 \frac{\omega}{r}=-\tan\frac{\omega T}{2}, \ \ \ \ \omega \ge 0.
\end{equation}


Note that $\cot \frac{\omega T}{2}$ and $-\tan \frac{\omega T}{2}$ are both periodical on $(-\infty, \infty)$ with the same period $\frac{2 \pi}{T}$. Further observe that both $\cot\frac{\omega T}{2}$ and
$-\tan\frac{\omega T}{2}$ have a single intersection point with the line $
\frac{\omega}r$ in each period, and that intersection points are roots of the equations. Denote by $\omega_j$ the $j$th (positive) root of Eq. \eqref{cot2} and by
$\tilde{\omega}_j$ the $j$th (positive) root of the transcendental equation (\ref{tan2}).
The relations of $\omega_j$ are as follows:
\begin{equation}\label{bdx}
   \tilde{\omega}_0=0<\omega_1<\frac{\pi}{T}<\tilde{\omega}_1<\frac{2\pi}{T}<\omega_2<\frac{3\pi}{T}<\tilde{\omega}_2<\frac{4\pi}{T}<\dots.
\end{equation}
From \eqref{e:ome}  and \eqref{bdx}, it is clear that all the eigenvalues are single


We conclude these results in the following proposition.
\begin{prop}\label{pop46}
 Let the operator $K_T$ be as in \eqref{kopera}-\eqref{kop1}.
Then $K_T$ only has point spectrum, and the eigenvalues are
\begin{equation}\label{eigen0}
\gamma_{T,j}=\left\{
      \begin{array}{ll}
       \frac{2r}{r^2+\tilde{\omega}_{(j-1)/{2}}^2} ,  & 2 \nmid j, \\
      \frac{2r}{r^2+{\omega}_{{j}/{2}}^2},  & 2 \mid j,
      \end{array}
\right.
\end{equation}
where $\omega_j, \, j>0 $ and $\tilde{\omega}_j,\,j\ge 0$ are the positive roots of the two transcendental equations
(\ref{cot2}) and (\ref{tan2}) respectively, such that
\begin{equation}\label{bdx1}
   \tilde{\omega}_0=0<\omega_1<\frac{\pi}{T}<\tilde{\omega}_1<\frac{2\pi}{T}<\omega_2<\frac{3\pi}{T}<\tilde{\omega}_2<\frac{4\pi}{T}<\dots.
\end{equation}
In particular, $\gamma_{T,1}=\frac{2}{r}$. Furthermore, the associated eigenfunction of $\gamma_{T,j}$ is
\begin{equation}\label{eigenfunc}
f_{j}(s)=\left\{
      \begin{array}{ll}
      e^{-\mi s d}[r\sin( s \tilde{\omega}_{(j-1)/2})+ \tilde{\omega}_{(j-1)/2}\cos( s \tilde{\omega}_{(j-1)/2})]   ,  & 2 \nmid j, \\
     e^{-\mi s d}[r\sin( s  {\omega}_{j/2})+  {\omega}_{j/2}\cos( s  {\omega}_{j/2})]  ,  & 2 \mid j.
      \end{array}
\right.
\end{equation}
\end{prop}
\begin{proof}
\eqref{eigen0} and \eqref{bdx1} are immediate from the analysis above and \eqref{e:ome}.
  Routine calculus based upon (\ref{kf11}), in combination with (\ref{solutionkr}) and (\ref{bdd0}), gives \eqref{eigenfunc}
\end{proof}
\begin{rem}
   The positive zeros of two transcendental equations (\ref{cot2}) and (\ref{tan2}) are similar to those of the famous Bessel functions
   $J_{\frac32}(z)=\sqrt{\frac{2}{\pi z}}(\frac{\sin z}{z}-\cos z)$ and $J_{-\frac32}(z)=-\sqrt{\frac{2}{\pi z}}(\sin z+\frac{\cos z}{z})$.
\end{rem}

\section{Proofs of the main results}
Recall that in Theorem \ref{main_thm 2} we defined
\begin{equation}\label{fl}
  F(\ell)=\int_0^{\infty} \log(1- \frac{2\ell r}{r^2+ \pi^2 y^2} )\dif y, \ \ \ \ \  \ell \in (-\infty, \frac r2).
\end{equation}
Let us now show that it is well defined. It follows from the well-known inequality
 \begin{equation}\label{lolog}
  x\le - \log(1-x)\le \frac{x}{1-x},\quad \forall x<1,
 \end{equation}
that we have
\begin{equation*}
  \abs{\log(1- \frac{2\ell r}{r^2+ \pi^2 y^2} )}\le \left\{
      \begin{array}{ll}
       \frac{2\abs{\ell} r}{r^2+ \pi^2 y^2}, & \ell \le 0, \\
       \frac{2\abs{\ell} r}{r(r-2\ell)+ \pi^2 y^2} ,& 0<\ell <\frac{r}{2},
             \end{array}
\right.
\end{equation*}
which implies that
\begin{align*}
  \abs{F(\ell)}& \le
   \left\{
      \begin{array}{ll}
  \int_0^{\infty}\frac{2\abs{\ell} r}{r^2+ \pi^2 y^2}\dif y=\abs{\ell},\quad  & \ell \le 0\\
   \int_0^{\infty}\frac{2\abs{\ell} r}{r(r-2\ell)+ \pi^2 y^2}\dif y=\sqrt{\frac{r}{r-2\ell}} {\ell} ,\quad& 0<\ell <\frac{r}{2}
        \end{array}
\right.
\\
&<\infty.
\end{align*}

To apply Gartner-Ellis Theorem, we shall make use of the following lemma.
\begin{lem}\label{lem51}
The function $F$ defined by Eq. \eqref{fl} is well defined. Moreover, we have
\begin{align}
    F'(\ell)=-\sqrt{\frac{r}{r-2\ell}}, \label{daoshu}
\end{align}
with $\lim\limits_{\ell\to \frac{r}{2}-}\abs{F'(\ell)}=\infty $.
\end{lem}
\begin{proof}
We only need to prove \eqref{daoshu}
Write
$$h(y, \ell):=\log(1- \frac{2\ell r}{r^2+ \pi^2 y^2} ),$$
for any $a\neq 0$, by the inequality (\ref{lolog}), we have
\begin{align*}
   \abs{ \frac{h(y,\ell+a)-h(y,\ell)}{a}}&=\abs{\frac{1}{a}\log(1-\frac{2ar}{r(r-2\ell)+\pi^2y^2})}\\
   &\le
    \left\{
      \begin{array}{ll}
    \frac{2 r}{r(r-2\ell)+\pi^2y^2},\quad &  a<0\\
   \frac{2 r}{\frac12 r(r-2\ell)+\pi^2y^2} ,\quad & 0<a <\frac{r-2\ell}{4},
        \end{array}
\right.
\end{align*}where both the right hand functions are integrable on $y\in [0,\infty)$.
Therefore, Lebesgue dominated convergence theorem yields
\begin{align*}
   F'(\ell)=\int_0^{\infty} \frac{\partial}{\partial \ell}h(y,\ell)\dif y =-\int_0^{\infty} \frac{2r}{r(r-2\ell)+\pi^2y^2}\dif y =-\sqrt{\frac{r}{r-2\ell}}.
\end{align*}
\end{proof}
\begin{prop}\label{prop42}
We write
$$\ell=\frac{2\sin^2\theta}{\cos\theta}\lambda(1+\lambda).$$
As long as
$\ell < \frac{r}{2}$, i.e.,
\Be
\lambda \in \left(-\frac 12-\frac 12 \sqrt{1+\frac{r \cos \theta}{\sin^2 \theta}},-\frac 12+\frac 12 \sqrt{1+\frac{r \cos \theta}{\sin^2 \theta}}\right),
\Ee
 we have
\Be
\lim_{T \rightarrow \infty} \frac 1T\log \E^{\mu} \exp \left(\ell \int_0^T |Z_t|^2 \dif t\right)= - \int_0^{\infty} \log(1- \frac{2\ell r}{r^2+ \pi^2 x^2} )  \dif x <\infty.
\Ee
As $\ell \ge \frac{r}{2}$, i.e.,
$$\lambda \notin \left(-\frac 12-\frac 12 \sqrt{1+\frac{r \cos \theta}{\sin^2 \theta}},-\frac 12+\frac 12 \sqrt{1+\frac{r \cos \theta}{\sin^2 \theta}}\right),$$ we have
\Be \label{e:LamInf}
\lim_{T \rightarrow \infty} \frac 1T\log \E^{\mu} \exp \left(\ell \int_0^T |Z_t|^2 \dif t\right)=\infty.
\Ee
\end{prop}
\begin{proof}
Since $\ell<\frac{r}{2}$, 
it follows from Corollary~\ref{cor3} that
\Be
\begin{split}
\frac{1}{T} \log \E^{\mu} \exp \left(\ell \int_0^T |Z_t|^2 \dif t\right)=-\frac{1}{T} \sum_{k \ge 1} \log(1- \ell \gamma_{T,k}).
\end{split}
\Ee
By \eqref{eigen0} and \eqref{bdx}, we have
\begin{equation*}
\sum_{k \ge 1} \log(1- \ell \frac{2r}{r^2+ \pi^2(\frac{k}{T})^2} ) \le \sum_{k \ge 2} \log(1- \ell \gamma_{T,k})\le \sum_{k \ge 0} \log(1- \ell  \frac{2r}{r^2+ \pi^2(\frac{k}{T})^2} )
\end{equation*}
for $\ell \le 0$, and
\begin{equation*}
\sum_{k \ge 0} \log(1- \ell  \frac{2r}{r^2+ \pi^2(\frac{k}{T})^2} )\le \sum_{k \ge 2} \log(1- \ell \gamma_{T,k})\le \sum_{k \ge 1} \log(1- \ell \frac{2r}{r^2+ \pi^2(\frac{k}{T})^2} )
\end{equation*}
for $\ell \in (0, \frac r2)$.

As $T\to\infty$, by Lemma~\ref{lem51} we have
\begin{align*}
\frac{1}{T}  \sum_{k \ge 1} \log(1- \ell \frac{2r}{r^2+ \pi^2(\frac{k}{T})^2} )& \to  \int_0^{\infty} \log(1- \frac{2\ell r}{r^2+ \pi^2 y^2} )  \dif y,\\
\frac{1}{T}   \sum_{k \ge 0} \log(1- \ell  \frac{2r}{r^2+ \pi^2(\frac{k}{T})^2} )& \to  \int_0^{\infty} \log(1- \frac{2\ell r}{r^2+ \pi^2 y^2} )  \dif y.
\end{align*}
Therefore, as $T\to\infty$,
\begin{equation*}
  -\frac{1}{T} \sum_{k \ge 1} \log(1- \ell \gamma_{T,k})\to - \int_0^{\infty} \log(1- \frac{2\ell r}{r^2+ \pi^2 y^2} )  \dif y.
\end{equation*}

Now we show the second limit. Since $  \ell  \ge \frac{r}{2}=\left( \gamma_{T,1}\right)^{-1}$, it follows from Corollary~\ref{cor3} that for any $T> 0$,
\
\Be
\frac{1}{T}\log \E^{\mu} \exp \left(\ell \int_0^T |Z_t|^2 \dif t\right)=\infty,
\Ee
which immediately gives the second limit.
\end{proof}

\noindent{\it Proofs of Theorems~\ref{main_thm} and \ref{main_thm 2}.\,}
Recall the relation \eqref{crm func 1}.
\eqref{crm func 000} follows from Proposition \ref{prop42}. Since $ \ell(\lambda)=\frac{2\sin^2\theta}{\cos\theta}\lambda(1+\lambda)$, we have that
$ \ell(\lambda)= \ell(-(1+\lambda))$ and
\begin{align*}
  \Lambda(\lambda)  & =  \left\{
      \begin{array}{ll}
     -F(\ell(\lambda)), &\quad \text{for } \lambda \in \left(-\frac 12-\frac 12 \sqrt{1+\frac{r \cos \theta}{\sin^2 \theta}},-\frac 12+\frac 12 \sqrt{1+\frac{r \cos \theta}{\sin^2 \theta}}\right) ,  \\
     \infty , &\quad \text{for } \lambda \notin \left(-\frac 12-\frac 12 \sqrt{1+\frac{r \cos \theta}{\sin^2 \theta}},-\frac 12+\frac 12 \sqrt{1+\frac{r \cos \theta}{\sin^2 \theta}}\right),
      \end{array}
\right.\\
    &=\left\{
      \begin{array}{ll}
     -F(\ell(-(1+\lambda))), &\quad \text{for } \lambda \in \left(-\frac 12-\frac 12 \sqrt{1+\frac{r \cos \theta}{\sin^2 \theta}},-\frac 12+\frac 12 \sqrt{1+\frac{r \cos \theta}{\sin^2 \theta}}\right) ,  \\
     \infty , &\quad \text{for } \lambda \notin \left(-\frac 12-\frac 12 \sqrt{1+\frac{r \cos \theta}{\sin^2 \theta}},-\frac 12+\frac 12 \sqrt{1+\frac{r \cos \theta}{\sin^2 \theta}}\right),
      \end{array}
\right.\\
&=\Lambda(-(1+\lambda)),
\end{align*}
Since $\theta \in \left(-\frac{\pi}2,0\right) \cup \left(0,\frac{\pi} 2\right)$, Proposition~\ref{prop42} implies
\begin{equation*}
  \mathcal{D}_{\lambda}\triangleq\set{\lambda: \Lambda(\lambda)<\infty}=\left(-\frac 12-\frac 12 \sqrt{1+\frac{r \cos \theta}{\sin^2 \theta}},-\frac 12+\frac 12 \sqrt{1+\frac{r \cos \theta}{\sin^2 \theta}}\right),
\end{equation*}
 From Lemma~\ref{lem51}, for $\lambda\in \mathcal{D}_{\lambda}$ we have
\begin{equation*}
   \Lambda'(\lambda)=\frac{2\sin^2\theta}{\cos\theta}(1+2\lambda)\sqrt{\frac{r}{r-2\ell}}.
\end{equation*}
 As $\lambda\in \partial \mathcal{D}_{\lambda}$, $\ell \rightarrow \frac r2-$ and thus $\Lambda'(\lambda) \rightarrow \infty$. Hence, the Cram\'{e}r function $\Lambda(\cdot)$ is essentially smooth.

 By G\"{a}rtner-Ellis theorem \cite{demzeit} implies that the LDP holds with the good rate functions $I(x)$ such that
\begin{align*}
  I(x) &=\sup\set{x\lambda +F(\ell(\lambda)):\, \lambda \in \left(-\frac 12-\frac 12 \sqrt{1+\frac{r \cos \theta}{\sin^2 \theta}},-\frac 12+\frac 12 \sqrt{1+\frac{r \cos \theta}{\sin^2 \theta}}\right)} \\
    &=\sup \set{x\lambda(\ell)+F(\ell):\,-\frac{\sin^2\theta}{2\cos\theta}\le \ell<\frac{r}{2}},
\end{align*}
where $F(\ell)$ is given as (\ref{fl}) and $\lambda(\ell)$ is the inverse function of $
   \ell=\frac{2\sin^2\theta}{\cos\theta}\lambda(1+\lambda)$.
By differentiating the function $s(\ell):=x\lambda(\ell)+F(\ell)$, we have that the unique zero point is
\begin{equation*}
  \ell_0=\frac{(x^2c^2-1)r}{2(x^2c+2r)c}
\end{equation*}
with $c=\frac{\cos \theta}{2 \sin^2 \theta}$.
Substituting it into $s(\ell)$, we obtain (\ref{rate func}) since $\lim\limits_{\ell\to \frac{r}{2}-}F(\ell)=-\infty$.

We have proved above
$$\Lambda(\lambda)=\Lambda(-(1+\lambda)),$$
which implies that $I(x)=I(-x)-x$ similar to the proof
of  \cite[Theorem 2.4]{jiangzhang}. We can also derive $ I(x)=I(-x)-x$
from \eqref{rate func} by a direct but complicated calculation.
{\hfill\large{$\Box$}}
\vskip 0.2cm {\small {\bf  Acknowledgements}\   This work was
partly supported by  NSFC(No.11101137) and Hunan Provincial NSFC(No. 2015JJ2055). H. Ge is supported by NSFC(No.21373021), and the Foundation for Excellent Ph.D. Dissertation from
the Ministry of Education in China (No. 201119). J. Xiong was  supported by Macao Science and Technology Fund FDCT 076/2012/A3 and Multi-Year Research Grants of the University of Macau Nos. MYRG2014-00015-FST and MYRG2014-00034-FST. L. Xu is supported by the grant SRG2013-00064-FST, the grant Science and Technology Development Fund, Macao S.A.R FDCT  049/2014/A1 and the grant MYRG2015-00021-FST.

\end{document}